\UseRawInputEncoding
\documentclass[10pt]{article}

\usepackage[dvips]{color}
\usepackage{epsfig}
\usepackage{float,amsthm,amssymb,amsfonts}
\usepackage{ amssymb,amsmath,graphicx, amsfonts, latexsym}
\usepackage{hyperref}
\usepackage{array,multirow}
\usepackage[fleqn]{mathtools}

\begin{document}
\theoremstyle{plain}
\newtheorem{theorem}{Theorem}[section]
\newtheorem{corollary}[theorem]{Corollary}
\newtheorem{lemma}[theorem]{Lemma}
\newtheorem{proposition}[theorem]{Proposition}
\newtheorem{remark}[theorem]{Remark}

\theoremstyle{definition}
\newtheorem{defn}{Definition}
\newtheorem{definition}[theorem]{Definition}
\newtheorem{example}[theorem]{Example}
\newtheorem{conjecture}[theorem]{Conjecture}

\title{On the Ranks of Semigroup of Order-preserving or Order-reversing Partial Contraction Mappings on a Finite Chain}
\author{\bf B. Ali, M. A. Jada ~and M. M. Zubairu\footnote{Corresponding Author. ~~Email: $mmzubairu.mth@buk.edu.ng$} \\[2mm]
\it\small Department of Mathematical Sciences, Nigerian Defence Academy, P.M.B. 2109, Kaduna, Nigeria\\
\it\small  \texttt{bali@nda.edu.ng}\\[3mm]
\it\small Department of Mathematics, Sule Lamido University, P.M.B. 048, Kafin Hausa, Nigeria\\
\it\small  \texttt{muhammada.jada@slu.edu.ng}\\[3mm]
\it\small  Department of Mathematics, Bayero  University Kano, P.M.B. 3011, Kano, Nigeria\\
\it\small  \texttt{mmzubairu.mth@buk.edu.ng}\\
}
\date{}
\maketitle\

\begin{abstract}
 Let $\mathcal{CP}_n$ be the semigroup of partial contraction mappings on $[n]=\{1,2,\ldots,n\}$ and let $\mathcal{OCP}_n$ and $\mathcal{ORCP}_n$ be its subsemigroups consisting of all order-preserving and of all order-preserving or order-reversing, partial contraction mappings, respectively. In this paper we obtain the rank of the two semigroups, $\mathcal{OCP}_n$ and $\mathcal{ORCP}_n$, respectively.
 \end{abstract}


\section{Introduction}
Let $[n]$ denote a finite chain, $\{1,2, \ldots ,n\}$. A \emph{partial transformation} is a map $\alpha: \textrm{dom }\alpha\subseteq [n] \rightarrow \textrm{im }\alpha\subseteq [n]$. Denote $\mathcal{P}_{n}$ to be the \emph{semigroup of all partial transformation} on $[n]$. A map $\alpha\in \mathcal{P}_{n}$ is called a \emph{total} or \emph{full} transformation if $\textrm{dom }\alpha=[n]$. The collection of all full transformations on $[n]$ is known to be the \emph{semigroup of full transformation} denoted by $\mathcal{T}_n$.
For $\alpha\in \mathcal{P}_{n}$, $\alpha$ is said to be \emph{order-preserving} (respectively, \emph{order reversing}) if  (for all $x,y \in \textrm{dom }\alpha$) $x\leq y$ implies $x\alpha\leq y\alpha$ (respectively, $x\alpha\geq y\alpha$); an \emph{isometry} (i. e., \emph{ distance preserving}) if (for all $x,y \in \textrm{dom }\alpha$) $|x\alpha-y\alpha|=|x-y|$; and a \emph{contraction} if (for all $x,y \in \textrm{dom }\alpha$) $|x\alpha-y\alpha|\leq |x-y|$.
The collection of all contraction mappings on $[n]$ is called the \emph{semigroup of partial  contraction mappings} denoted by $\mathcal{CP}_{n}$.

Let
\begin{equation}\label{ocp} \mathcal{OCP}_{n}=\{\alpha\in \mathcal{CP}_{n}:(\textnormal{for all } x,y\in \textnormal{dom }\alpha)~x\leq y~ \Rightarrow~ x\alpha\leq y\alpha\} \end{equation} and
\begin{equation}\label{orcp} \mathcal{ORCP}_{n}= \mathcal{OCP}_{n} \cup \{\alpha\in \mathcal{CP}_{n}:(\textnormal{for all } x,y\in \textnormal{dom }\alpha)~x\leq y~ \Rightarrow~ x\alpha\geq y\alpha\}. \end{equation}
Then, $\mathcal{OCP}_n$ and $\mathcal{ORCP}_n$ are subsemigroups of $\mathcal{CP}_n$ known as the \emph{semigroup of order-preserving partial contraction mappings} and the \emph{semigroup of order-preserving or order-reversing partial contraction mappings}, respectively.

 A non-empty subset $A$ of a semigroup $S$ is said to be a \emph{generating set} of $S$ (denoted as $\langle A\rangle=S$) if for every $x\in S$, $x$ can be written as finite product of some elements in $A$. If $S$ is generated by a finite subset, then, the smallest subset that generate $S$ is called the \emph{minimal generating set} of $S$ while its cardinality is called the \emph{rank} of $S$, denoted by rank($S$). That is, rank$(S)=\min\{|A|: A\subseteq S \textnormal{ and } \langle A\rangle=S\}$.

The importance of the study of rank properties in semigroup theory cannot be over emphasized. Various scholars study rank properties on the semigroup $\mathcal{P}_{n}$ and many of its subsemigroups over the years, and fruitful results were recorded, see for example, \cite{Ayk,Gr,gu2,goms,H2}.

The semigroup of partial contraction mappings being a relatively new semigroup with promising research problems attracts the attention of many researchers in the area. For example, Adeshola in~\cite{ade}, investigated some algebraic and combinatorial properties of some subsemigroups of $\mathcal{CT}_n$ (the semigroup of full contraction mappings) and obtain the cardinalities of some of its subsemigroups. The study of Green's and starred Green's relations of $\mathcal{CT}_n$ was presented in Garba \emph{et al.,}~\cite{jm}. The results in~\cite{jm} were extended in~\cite{zb2} to give complete characterization of the Green's relations on $\mathcal{CT}_{n}$. Many more algebraic and combinatorial studies on $\mathcal{CP}_{n}$ and some of its subsemigroup were conducted see,~\cite{zb,zb1,zao} which by no means a complete list.

However, a lot of questions regarding the rank properties of $\mathcal{CP}_{n}$ and many of its subsemigroups are still unanswered. Recently, Toker examined the ranks of the semigroup of \emph{order-preserving or order-reversing full contractions}, $\mathcal{ORCT}_{n}$ and that of its subsemigroup $\mathcal{OCT}_n$, consisting of all order-preserving elements in~\cite{tok}. Bugay~\cite{bugay} extended the work in \cite{tok} by considering the ideals of the semigroups $\mathcal{ORCT}_{n}$ and $\mathcal{OCT}_n$ respectively. In~\cite{jd}, the study of nilpotent elements in $\mathcal{OCP}_{n}$ and its subsemigroup $\mathcal{OCI}_{n}$, consisting of all partial one-to-one transformations was presented and the ranks of the subsemigroups generated by the nilpotents was obtained. Also, in~\cite{gu} the study of algebraic and rank properties of the semigroup $\mathcal{OCI}_{n}$ was conducted. Umar and Alkharousi study the rank and combinatorial properties of the semigroups of partial isometries in~\cite{ak1,ak2}. The study of the rank properties of isometries were further exploited in~\cite{me,ayik}.

Having studied (by Toker~\cite{tok} and Bugay~\cite{bugay}) the rank properties of the semigroups, $\mathcal{ORCT}_{n}$ and $\mathcal{OCT}_n$, in this paper we consider a larger semigroup, $\mathcal{ORCP}_{n}$ and its subsemigroup, $\mathcal{OCP}_{n}$, and investigated their ranks. The structure of the paper is as follows. After this introductory section, followed is the preliminary section. In section $3$, we present the rank of $\mathcal{OCP}_{n}$, and show that the semigroup $\mathcal{OCP}_{n}$ has rank $2n$ for all $n\geq 3$. The results in section $3$ were extended in section $4$ and shown that, despite the semigroup $\mathcal{ORCP}_{n}$ is larger, it also has the same rank as $\mathcal{OCP}_{n}$.


\section{Preliminaries}
In this section, we give definition of some basic terms. We also quote some results from related literature that will be needed in proving some results in the paper.

Let $\alpha$ be an element of $\mathcal{CP}_{n}$. Denote $\textrm{dom }\alpha$, $\textrm{im }\alpha$, and  $h(\alpha)$ to be the domain of $\alpha$, image of $\alpha$, and $|\textrm{im }\alpha|$, respectively. The kernel of $\alpha$ is denoted by ker~$\alpha$, and defined as  ker~$\alpha=\{(x,y)\in \textnormal{dom }\alpha\times \textnormal{dom }\alpha: x\alpha=y\alpha\}$.
  Furthermore, for $\alpha,\beta \in \mathcal{CP}_{n}$, the composition of $\alpha$ and $\beta$ is defined as $x(\alpha \circ \beta) =((x)\alpha)\beta$ for all $x$ in dom~$\alpha$.  Without ambiguity, we shall be using the notation $\alpha\beta$ to denote $\alpha \circ\beta$ and $x\alpha$ to serve as the image of $x$ under $\alpha$ instead of the notation $\alpha(x)$ so that our composition of maps reads from left to right, i. e., $x\alpha\beta$ to serve as $\beta(\alpha(x))$. Also, a non-empty subset $A$ of $[n]$ is called convex if for every $x,y \in A$, $x< z < y$ implies $z\in A$. We shall write $1_A$ for any subset $A\subseteq [n]$ to denote partial identity map on $A$.

 Next, given any transformation $\alpha$ in $\mathcal{ORCP}_n$, $\alpha$ can be expressed as \begin{equation}\label{alf} \alpha=\left(\begin{array}{cccc}
            A_1 & A_2 & \ldots & A_p \\
            b_1 & b_2 & \ldots & b_p
          \end{array} \right)\, (1\leq p\leq n). \end{equation}
 The blocks $A_i (1\leq i\leq p)$ are equivalence classes under the relation ker~$\alpha$, and the collection of all the equivalence classes of the relation ker~$\alpha$  partitioned the domain of $\alpha$ called the \emph{kernel partition} of $\alpha$, denoted by kp$(\alpha)$, i. e., $[n]\backslash \textnormal{ker }\alpha = \textnormal{kp}(\alpha)$. Moreover, kp$(\alpha)$ is ordered under the usual ordering, that is, $A_1<A_2< \cdots < A_p$, where $A_i<A_j$ if and only if $a<b$ for all $a\in A_i$ and $b\in A_j$. Further, for each block $A_i$, if $|A_i|\geq 2$ then $a_i<a_{i+1}\in A_i$. In addition, if $\alpha$ is order-preserving then $b_1<b_2<\ldots <b_p$, and if $\alpha$ is order-reversing then $b_p<b_{p-1} <\ldots <b_1$. Thus, for the rest of the content of the work we shall consider $\alpha\in\mathcal{ORCP}_n$ to be as expressed in equation~\eqref{alf} unless otherwise specified.

Let $S$ be a semigroup and $x\in S$. Then, $Sx$, $xS$, and $SxS$ are principal ideals generated by $x$ in $S$. They are known as the \emph{principal left, principal right}, and \emph{principal two sided} ideal, respectively generated by $x$. Green's relations are equivalence relations defined on a semigroup based on the principal ideals its elements generate. These relations are vital in understanding the algebraic structure of any semigroup more especially regular semigroups. There are five of these equivalences, namely, $\mathcal{L}$, $\mathcal{R}$, $\mathcal{J}$, $\mathcal{D}$ and $\mathcal{H}$ relation defined as follows:  For all $a,b\in S$,
\begin{eqnarray*}
    (a,b)\in \mathcal{L} &\Longleftrightarrow& S^{1}a= S^{1}b; \\
    (a,b)\in \mathcal{R} &\Longleftrightarrow& aS^{1}= bS^{1}; \\
    (a,b)\in \mathcal{J} &\Longleftrightarrow& S^{1}a S^{1}= S^{1}b S^{1}.
\end{eqnarray*} The symbol $S^1$ denote a semigroup $S$ with adjoint identity if $S$ does not have one. $\mathcal{L}$ and $\mathcal{R}$ always commute and their composition gives a relation $\mathcal{D}$ (i. e., $\mathcal{L}\circ \mathcal{R} = \mathcal{R} \circ \mathcal{L}= \mathcal{D}$), while their intersection gives a relation $\mathcal{H}$ (i. e., $\mathcal{L}\cap\mathcal{R} = \mathcal{H}$). If a semigroup is non-regular, then, there is also need to understand its starred Green's relation (which is the generalization of the Green's relations) in order to classify such semigroup or study its rank properties. The starred Green's equivalences are also five, they are $\mathcal{L}^*$, $\mathcal{R}^*$, $\mathcal{J}^*$, $\mathcal{D}^*$, and $\mathcal{H}^*$ defined as follows: Given any semigroup $S$, $a\mathcal{L}^* b $ (resp., $a\mathcal{R}^* b $) if and only if $a$ and $b$ are Green's $\mathcal{L}$-related (resp., Green's $\mathcal{R}$-related) in some over-semigroup of $S$. $\mathcal{D}^*$ is a meet of $\mathcal{L}^*$ and $\mathcal{R}^*$, while $\mathcal{H}^*$ is their joint. For more properties of Green's relation and any other unexplained term we refer the reader to \cite{grn,howi,ph}.

Below are some results concerning starred Green's relations on the semigroups of $\mathcal{OCP}_n$ and $\mathcal{ORCP}_n$ from \cite{zb1}, which shall be needed in the subsequent sections.

\begin{theorem}\label{zbt}\cite[Theorem 1]{zb1} Let $\alpha,\beta\in S$ (where $S\in \{\mathcal{CP}_n, \mathcal{ORCP}_n,\mathcal{OCP}_n\})$. Then,
\begin{itemize}
 \item[(i)] $\alpha \mathcal{L}^* \beta$ if and only if $\textnormal{im }\alpha=\textnormal{im }\beta$.
 \item[(ii)] $\alpha \mathcal{R}^* \beta$ if and only if $\textnormal{ker }\alpha=\textnormal{ker }\beta$.
 \item[(iii)] $\alpha \mathcal{H}^* \beta$ if and only if $\textnormal{ker }\alpha=\textnormal{ker }\beta$ and $\textnormal{im }\alpha=\textnormal{im }\beta$.
 \item[(iv)] $\alpha \mathcal{D}^* \beta$ if and only if $|\textnormal{im }\alpha|=|\textnormal{im }\beta|$.
\end{itemize}
\end{theorem}

The following lemmas are also needed in proving Lemma~\ref{jr} and~\ref{jp}, respectively.

 Let $ D_p^*=\{\alpha \in \mathcal{OCT}_n: |\textnormal{im }\alpha|=p\}$, and let $ F_p^*=\{\alpha \in \mathcal{ORCT}_n: |\textnormal{im }\alpha|=p\}$. Then,

\begin{lemma}\label{tk1}\cite[Lemma 3.1]{tok} If $\alpha \in D_p^*$ then $\alpha \in \langle D_{p+1}^*\rangle$ for all $1\leq p\leq n-2$.
\end{lemma}

\begin{lemma}\label{tk2}\cite[Lemma 4.1]{tok} If $\alpha \in F_p^*$ then $\alpha \in \langle F_{p+1}^*\rangle$ for all $1\leq p\leq n-2$.
\end{lemma}


\section{Rank of $\mathcal{OCP}_n$}
   In this section, we compute the rank of the semigroup of order-preserving partial contraction mappings, $\mathcal{OCP}_n$.

  For $1\leq p\leq n$, let
\begin{equation}\label{kp} K_p=\{\alpha \in \mathcal{OCP}_n: |\textnormal{im }\alpha|=p\}\end{equation}
    and
 \begin{equation}\label{idl} \mathcal{OCP}_{n,p}=\{\alpha \in \mathcal{OCP}_n: |\textnormal{im }\alpha|\leq p\}. \end{equation}
 Observe that $K_p \subseteq \mathcal{OCP}_{n,p}$ for all $p$.

 We now proof the following lemma which is indispensable in finding the rank of $\mathcal{OCP}_{n,n-1}$.
\begin{lemma}\label{jr} $K_p\subseteq \langle K_{p+1}\rangle$, for all $1\leq p\leq n-2$.
\end{lemma}
\begin{proof} Let $\alpha \in K_p$ be as expressed in equation~\eqref{alf} where $1\leq p\leq n-2$. Then, we consider the following cases:

{\bf Case I:}  Suppose $\textnormal{im }\alpha$ is convex. If $\textnormal{dom }\alpha=[n]$, then $\alpha\in \mathcal{OCT}_n$ and the result follows from Lemma~\ref{tk1}. Now, suppose $\textnormal{dom }\alpha \subset [n]$ and let $c$ be in $[n]\backslash \textnormal{dom }\alpha$.
 If $c< \min A_1$ and $b_p<n$, then define $\beta$ and $\gamma$, respectively as:
 \begin{equation}\label{m1}
        \left(\begin{array}{ccccc}
            c & A_1 & A_2 & \cdots & A_p \\
            1  &  2 &  3 &  \cdots &  p+1
          \end{array} \right) \textrm{ and } \left(\begin{array}{cccccccc}
            2  &  3   & \cdots &  p+1 & p+2  \\
            b_1 & b_2 & \cdots & b_p & b_p+1
          \end{array} \right);\end{equation}
 if $c< \min A_1$ and $b_p=n$ (notice that since $\textnormal{im }\alpha$ is convex and $|\textnormal{im }\alpha|\leq n-2$, then $b_p=n$ would imply that $b_1>2$), then define $\beta$ and $\gamma$, respectively as:
 \begin{equation}\label{m2}
        \left(\begin{array}{ccccc}
            c & A_1 & A_2 & \cdots & A_p \\
            2  &  3 &  4 &  \cdots &  p+2
          \end{array} \right)
\textrm{ and }\left(\begin{array}{cccccccc}
            1  &  3   & 4 & \cdots &  p+2   \\
            b_1-1 & b_1& b_2 & \cdots & b_p
          \end{array} \right);\end{equation}
if $\max A_i<c< \min A_{i+1}$ ($i\in\{1,2,\ldots,p-1\}$) and $b_p<n$, then define $\beta$ and $\gamma$, respectively as:
 \begin{eqnarray}\label{m3}
        \left(\begin{array}{ccccccc}
            A_1  & \cdots & A_i& c & A_{i+1} & \cdots & A_p \\
            1  & \cdots &  i & i+1 &  i+2  & \cdots &  p+1
          \end{array} \right)
\textrm{ and }  \nonumber \\ \left(\begin{array}{ccccccc}
            1   & \cdots &  i &  i+2  & \cdots &  p+1 & p+2  \\
            b_1  & \cdots & b_i& b_{i+1} & \cdots & b_p & b_p+1
          \end{array} \right);\end{eqnarray}
if $\max A_i<c< \min A_{i+1}$ ($i\in\{1,2,\ldots,p-1\}$) and $b_p=n$, then define $\beta$ and $\gamma$, respectively as:
 \begin{eqnarray}\label{m4}
        \left(\begin{array}{ccccccc}
            A_1  & \cdots & A_i& c & A_{i+1} & \cdots & A_p \\
            2  & \cdots &  i+1 & i+2 &  i+3  & \cdots &  p+2
          \end{array} \right)
\textrm{ and }\nonumber \\\left(\begin{array}{ccccccc}
            1  &  2  &  \cdots &  i+1 &  i+3  & \cdots & p+2  \\
            b_1-1 & b_1 & \cdots & b_i& b_{i+1} & \cdots & b_p
          \end{array} \right);\end{eqnarray}

  if $\min A_i<c< \max A_{i}$ ($i\in\{1,2,\ldots,p\}$) and $b_p<n$, then define $\beta$ and $\gamma$, respectively as:
 \begin{eqnarray}\label{m31}
        \left(\begin{array}{ccccccc}
            A_1  & \cdots & \{\min A_i,\ldots,c-1\}& \{c+1,\ldots,\max A_i\}& A_{i+1} & \cdots & A_p \\
            1  & \cdots &  i & i+1 & i+2 & \cdots &  p+1
          \end{array} \right)
\textrm{ and }  \nonumber \\ \left(\begin{array}{cccccc}
            1   & \cdots &  \{i,i+1\}  & \cdots &  p+1 & p+2  \\
            b_1  & \cdots & b_i  & \cdots & b_p & b_p+1
          \end{array} \right);\end{eqnarray}
if $\min A_i<c< \max A_{i}$ ($i\in\{1,2,\ldots,p\}$) and $b_p=n$, then define $\beta$ and $\gamma$, respectively as:
 \begin{eqnarray}\label{m41}
        \left(\begin{array}{cccccc}
            A_1  & \cdots & \{\min A_i,\ldots,c-1\}& \{c+1,\ldots,\max A_i\}  & \cdots & A_p \\
            2  & \cdots &  i+1 & i+2  & \cdots &  p+2
          \end{array} \right)
\textrm{ and }  \nonumber \\ \left(\begin{array}{ccccccc}
            1  & 2  & \cdots &  \{i+1,i+2\} & i+3  & \cdots  & p+2  \\
            b_1-1 & b_1  & \cdots & b_i  & b_{i+1}& \cdots & b_p
          \end{array} \right);\end{eqnarray}
if $\max A_p<c$ and $b_p<n$, then define $\beta$ and $\gamma$, respectively as:
 \begin{equation}\label{m5}
        \left(\begin{array}{ccccc}
         A_1 & A_2 & \cdots & A_p & c \\
         1  &  2  &  \cdots & p&  p+1
          \end{array} \right) \textrm{ and } \left(\begin{array}{cccccccc}
            1  &  2   & \cdots &  p & p+2  \\
            b_1 & b_2 & \cdots & b_p & b_p+1
          \end{array} \right);\end{equation}
 and if $\max A_p <c$ and $b_p=n$, then define $\beta$ and $\gamma$, respectively as:
 \begin{equation}\label{m6}
        \left(\begin{array}{ccccc}
         A_1 & A_2 & \cdots & A_p & c\\
          2  &  3 &  \cdots &  p+1& p+2
          \end{array} \right)
\textrm{ and }\left(\begin{array}{cccccccc}
            1  &  2  & 3 & \cdots &  p+1   \\
            b_1-1 & b_1& b_2 & \cdots & b_p
          \end{array} \right).\end{equation}
           Then, it is not difficult to verify that $\beta,\gamma\in K_{p+1}$ in all the cases and $\beta\gamma=\alpha$.

{\bf Case II:} Suppose $\textnormal{im }\alpha$ is not convex, i. e., there exists $i\in\{1,2,\ldots,p-1\}$ such that $b_{i+1}-b_i\geq 2$. Then, since $\alpha$ is a contraction, we must have $\min A_{i+1}-\max A_i\geq 2$ . Define
$$\beta=\left(\begin{array}{cccccccc}
            A_1 & A_2 & \cdots & A_i& \max A_i +1 & A_{i+1} & \cdots & A_p \\
            b_1 & b_2 & \cdots & b_i& b_i+1 & b_{i+1} & \cdots & b_p
 \end{array} \right).$$ Recall that $|\textnormal{im }\alpha|\leq n-2$, therefore $|\textnormal{im }\beta|\leq n-1$. Let $x\in [n]\backslash \textnormal{im }\beta$ and let $A=\textnormal{im }\beta \cup \{x\}$. Define $\gamma$ to be the partial identity on $A$. Then, it is easy to see that $\beta$ and $\gamma$ are in $K_{p+1}$ and $\alpha=\beta\gamma$.
\end{proof}

\begin{corollary}\label{cjr} $\langle K_{n-1}\rangle=\mathcal{OCP}_{n,n-1}$ for all $n\geq 3$.
\end{corollary}
\begin{proof} It follows Lemma~\ref{jr} and the fact that $\langle K_{n-1}\rangle \leq\mathcal{OCP}_{n,n-1}$.
\end{proof}

An element $\alpha$ in $\mathcal{OCP}_n$ (as well as $\mathcal{ORCP}_n$) is said to have a \emph{projection characteristic} $(k,p)$ or belong to the set $[k,p]$ if $|\textnormal{dom }\alpha|=k$ and $|\textnormal{im }\alpha|=p$ (see \cite{gu2}). Now, let $\alpha \in K_{n-1}$ where $n\geq 3$. Then, $|\textnormal{dom }\alpha|$ is either $n$ or $n-1$, i. e., either $\alpha$ belongs to $[n,n-1]$ or $[n-1,n-1]$. Clearly if $\alpha\in [n,n-1]$ then $\alpha$ is a full contraction and its kernel partition is of the form $\{ \{1\},\{2\},\ldots,\{i,i+1\},\ldots,\{n-1\},\{n\}\}$ ($1\leq i\leq n-1$). On the other hand, if $\alpha\in [n-1,n-1]$, then $\alpha$ is one-to-one, and so kp$(\alpha)$ is just a partition of $[n]\backslash \{i\}$ ($i=1,2,\ldots,n$) into singleton classes.

Moreover, if $\alpha\in K_{n-1}$ and $\textnormal{dom }\alpha$ is convex, then by contractive property of $\alpha$, $\textnormal{im }\alpha$ must also be convex (see also Ali \emph{et. al.,}~\cite{zb1}). Therefore, $\textnormal{im }\alpha$ will either be $\{1,2,\ldots,n-1\}$ or $\{2,3,\ldots,n\}$. But if $\textnormal{dom }\alpha$ is not convex, i. e., if kernel partition of $\alpha$ is of the form $\{\{1\},\ldots,\{j-1\},\{j+1\},\{j+2\},\ldots,\{n\}\}$ ($2\leq j\leq n-1$), then $\textnormal{dom }\alpha$ has three possible choice of images: $\{1,2,\ldots,n-1\}$, $\{2,3,\ldots,n\}$ or $\{1,\ldots,j-1,j+1,\ldots,n\}$.

For $1\leq s,t\leq n$ and $1\leq i\leq n-1$, denote $P_{[n]\backslash \{s\}}$ to be a partition of $[n]\backslash \{s\}$ into singleton classes; and $P_{\{i, i+1\}}$ to be a partition of $[n]$ into $n-1$ classes as $\{ \{1\},\ldots,\{i-1\},\{i,i+1\},\{i+2\},\ldots,\{n\}\}$. Also, denote $\alpha_{s,t}$ to mean an element in $[n-1,n-1]$ whose kernel partition is $P_{[n]\backslash \{s\}}$  and image, $[n]\backslash \{t\}$. Further, denote $\pi_{(i,i+1),k}$ ($k=1,2$) to mean an element in $[n,n-1]$ whose kernel partition is $P_{\{i, i+1\}}$ and image, $[n]\backslash \{k\}$. For the purpose of illustration, consider the transformations below.\\
  $\alpha_{n,1}=\left(\begin{array}{cccc}
         1 & 2 & \cdots & n-1 \\
          2  &  3 &  \cdots &  n
          \end{array} \right)$, $\alpha_{2,2}=\left(\begin{array}{ccccc}
           1  &  3 & 4 & \cdots &  n\\
          1 & 3 & 4& \cdots & n
          \end{array} \right)$ and $\pi_{(2,3),1}= \left(\begin{array}{cccc}
         1 & \{2,3\} & \cdots & n \\
          2  &  3 &  \cdots &  n
          \end{array} \right)$.

 Then, we have a tabular display of the elements of $K_{n-1}$ below, where the horizontal rows represent the $\mathcal{R}^*-$classes and vertical columns represent the $\mathcal{L}^*-$classes.
\begin{equation}
\begin{tabular}{|c|c|c|c|c||}
\hline
  {\bf Kernel partitions / Image sets} &$[n]\backslash \{1\}$  & $[n]\backslash \{j\}$ & $[n]\backslash \{n\}$ \\
  \hline
$P_{[n]\backslash \{1\}}$ & $\alpha_{1,1}$ &&$\alpha_{1,n}$    \\ \hline
$P_{[n]\backslash \{j\}}$ & $\alpha_{j,1}$ & $\alpha_{j,j}$ &$\alpha_{j,n}$     \\ \hline
$P_{[n]\backslash \{n\}}$ & $\alpha_{n,1}$ && $\alpha_{n,n}$    \\ \hline
 $P_{\{i, i+1\}}$ & $\pi_{(i,i+1),1}$ & &$\pi_{(i,i+1),n}$   \\
  \hline
\end{tabular} \tag{1}\label{tab1}
\end{equation}
where $1\leq i\leq n-1$ and $2\leq j\leq n-1$.

It is not difficult to see that the number of rows in Table~\eqref{tab1} is $2n-1$. It therefore follows that $K_{n-1}$ has $2n-1$ $\mathcal{R}^*-$classes. Thus, we have the following lemma.

\begin{lemma}\label{lr} For $n\geq 3$, \textnormal{rank(}$\mathcal{OCP}_{n,n-1})\geq 2n-1$.
\end{lemma}
\begin{proof} Since any generating set of $\langle K_{n-1}\rangle$ must contain at least one representative from each of the $\mathcal{R}^*-$classes, it follows that rank($\langle K_{n-1}\rangle)\geq 2n-1$ which implies by Corollary~\ref{cjr} that \textnormal{rank(}$\mathcal{OCP}_{n,n-1})\geq 2n-1$.
\end{proof}

\begin{proposition}\label{gs1} Let $ G=\{\alpha_{n,1},\alpha_{1,n}, \alpha_{j,j},\pi_{\{i,i+1\},n}: 2\leq j\leq n-1,\, 1\leq i\leq n-1 \}$. Then, $G$ is a minimal generating set of $\langle K_{n-1}\rangle$.
\end{proposition}
\begin{proof} Firstly, it is not difficult to see from the following multiplications,
\begin{equation*}
     \alpha_{1,n}\cdot\alpha_{n,1}= \left(\begin{array}{cccc}
         2 & 3  &\cdots & n\\
         1 & 2 & \cdots & n-1
          \end{array} \right) \left(\begin{array}{cccc}
         1 & 2  &\cdots & n-1\\
          2 & 3 & \cdots & n
          \end{array} \right)= \left(\begin{array}{cccc}
         2 & 3  &\cdots & n \\
         2 & 3 & \cdots & n
          \end{array} \right) =\alpha_{1,1}; \end{equation*}
           \begin{equation*}
         \alpha_{n,1}\cdot \alpha_{1,n}=  \left(\begin{array}{cccc}
         1 & 2  &\cdots & n -1\\
         2 & 3 & \cdots & n
          \end{array} \right) \left(\begin{array}{cccc}
         2 & 3  &\cdots & n\\
          1  & 2 & \cdots & n-1
          \end{array} \right)= \left(\begin{array}{cccc}
         1 & 2  &\cdots & n-1\\
         1 & 2 & \cdots & n-1
          \end{array} \right)= \alpha_{n,n}; \end{equation*}
           \begin{equation*}
         \pi_{\{i,i+1\},n}\cdot\alpha_{n,1}= \left(\begin{array}{ccccccc}
          1  & \cdots & i-1 & \{i,i+1\}& i+2 &\cdots & n \\
          1 & \cdots &  i-1 & i & i+1 & \cdots  & n-1
          \end{array} \right)\left(\begin{array}{cccc}
         1 & 2  &\cdots & n-1\\
          2 & 3 & \cdots & n
          \end{array} \right) \end{equation*}
          \begin{equation*}= \left(\begin{array}{ccccccc}
          1  & \cdots & i-1 & \{i,i+1\}& i+2 &\cdots & n \\
          2 & \cdots &  i & i+1 & i+2 & \cdots  & n
          \end{array} \right)= \pi_{\{i,i+1\},1};
         \end{equation*}
         \begin{equation*}
          \alpha_{j,j}\cdot\pi_{\{j,j+1\},n}= \left(\begin{array}{cccccc}
         1 & \cdots & j-1 & j+1 &\cdots & n\\
         1 & \cdots & j-1 & j+1 & \cdots & n
          \end{array} \right)  \left(\begin{array}{ccccccc}
          1  & \cdots & \{j,j+1\}& j+2 &\cdots & n \\
          1 & \cdots  & j & j+1 & \cdots  & n-1
          \end{array} \right)  \end{equation*}
          \begin{equation*}
         = \left(\begin{array}{cccccc}
         1 & \cdots & j-1 & j+1 &\cdots & n\\
         1 & \cdots & j-1 & j & \cdots & n-1
          \end{array} \right) =\alpha_{j,n}; \textnormal{ and} \end{equation*}
          \begin{equation*}
          \\ \alpha_{j,j}\cdot\pi_{\{j,j+1\},1}= \left(\begin{array}{cccccc}
         1 & \cdots & j-1 & j+1 &\cdots & n\\
         1 & \cdots & j-1 & j+1 & \cdots & n
          \end{array} \right)  \left(\begin{array}{cccccc}
          1  & \cdots  & \{j,j+1\}& j+2 &\cdots & n \\
          2 & \cdots  & j+1 & j+2 & \cdots  & n
          \end{array} \right) \end{equation*}
          \begin{equation*}
          = \left(\begin{array}{cccccc}
         1 & \cdots & j-1 & j+1 &\cdots & n\\
         2 & \cdots & j & j+1 & \cdots & n
          \end{array} \right) =\alpha_{j,1};
\end{equation*}            
 that $ K_{n-1}\subseteq \langle G\rangle$ and as such, since $G\subseteq K_{n-1}$, then $\langle K_{n-1}\rangle = \langle G\rangle$.

  Secondly, we notice that elements of $G$ are just representatives of the $\mathcal{R}^*-$classes in $K_{n-1}$. Therefore, $G$ is a minimal generating set of $\langle K_{n-1}\rangle$.
\end{proof}

\begin{corollary}\label{tr} For $n\geq3$, \textnormal{rank(}$\mathcal{OCP}_{n,n-1})=2n-1$.
\end{corollary}

\begin{proof} It is easy to see from Proposition~\ref{gs1} that the cardinality of the minimal generating set $G$, of $K_{n-1}$ is $2n-1$. Therefore, $\textnormal{rank}(\langle K_{n-1}\rangle)=2n-1$ which implies by Corollary~\ref{cjr} that $\textnormal{rank}(\mathcal{OCP}_{n,n-1})=2n-1$.
\end{proof}

Now we have the main result of this section.

\begin{theorem}\label{cr}
Let $\mathcal{OCP}_n$ be as defined in equation~\eqref{ocp}. Then, the rank of $\mathcal{OCP}_n$ is $2n$ for all $n\geq 3$.
\end{theorem}

\begin{proof} First, observe that the only element of $\mathcal{OCP}_n$ of height $n$ is the identity element, $1_{[n]}$. It thus follows that $\mathcal{OCP}_n= \mathcal{OCP}_{n,n-1}\cup \{1_{[n]}\}$. Notice also that since $\mathcal{OCP}_{n,n-1}$ is an ideal of $\mathcal{OCP}_n$, then rank$(\mathcal{OCP}_n)= 2n$, as required
\end{proof}


\section{Rank of $\mathcal{ORCP}_n$}
   In this section, we extend the results obtained in section 3 to compute the rank of the semigroup of order-preserving or order-reversing partial contraction mappings, $\mathcal{ORCP}_n$.

    Let
\begin{equation}\label{kp} W_p=\{\alpha \in \mathcal{ORCP}_n: |\textnormal{im }\alpha|=p\}\end{equation}
    and
 \begin{equation}\label{id2} \mathcal{ORCP}_{n,p}=\{\alpha \in \mathcal{ORCP}_n: |\textnormal{im }\alpha|\leq p\}, \end{equation}
where $1\leq p\leq n$. Observe that $K_p$ as defined in equation~\eqref{kp} contain all the order-preserving elements in $W_p$. Let $K_p^*$ denote the set of all order-reversing elements in $W_p$. Then, $W_p= K_p\cup K_p^*$. Now we have the following lemma which is an extension of Lemma~\ref{jr}.
\begin{lemma}\label{jp} For $1\leq p\leq n-2$, $W_p\subseteq \langle W_{p+1}\rangle$.
\end{lemma}
\begin{proof} Let $\alpha=\left(\begin{array}{cccc}
            A_1 & A_2 & \ldots & A_p \\
            b_1 & b_2 & \ldots & b_p
          \end{array} \right)$
be element of $W_p$ where $1\leq p\leq n-2$. If $|\textnormal{dom }\alpha|=n$, then $\alpha$ is a full contraction and the result follows from Lemma~\ref{tk2}. Now, suppose $|\textnormal{dom }\alpha|<n$. If $\alpha\in K_p$, then the result follows from Lemma~\ref{jr}, and if $\alpha\in K_p^*$, then we proceed by cases similar to what we had in the proof of Lemma~\ref{jr}.

{\bf Case I:} If $\textnormal{im }\alpha$ is convex. Then, since $\alpha$ is order-reversing it can be expressed as\\
$\left(\begin{array}{cccc}
            A_1 & A_2 & \cdots & A_p \\
            k & k-1 & \cdots & k-p+1
          \end{array} \right)$ ($p\leq k\leq n$). Let $c\in [n]\backslash \textnormal{dom }\alpha$.
 If $c< \min A_1$ and $p<k$, then define $\beta$ and $\gamma$, respectively as:
 \begin{equation}\label{n1}
        \left(\begin{array}{ccccc}
            c & A_1 & A_2 & \cdots & A_p \\
            1  &  2 &  3 &  \cdots &  p+1
          \end{array} \right) \textrm{ and } \left(\begin{array}{cccccccc}
            2  &  3   & \cdots &  p+1 & p+2  \\
            k & k-1 & \cdots & k-p+1 & k-p
          \end{array} \right);\end{equation}
 if $c< \min A_1$ and $p=k$, then $k\leq n-2$, and so $\beta$ and $\gamma$ can be defined respectively as:
 \begin{equation}\label{n2}
        \left(\begin{array}{ccccc}
            c & A_1 & A_2 & \cdots & A_p \\
            2  &  3 &  4 &  \cdots &  p+2
          \end{array} \right)
\textrm{ and }\left(\begin{array}{cccccccc}
            1  &  3 & 4 & \cdots &  p+2   \\
            k+1 & k &  k-1 &\cdots & k-p+1
          \end{array} \right);\end{equation}
if $\max A_i<c< \min A_{i+1}$ ($i\in\{1,2,\ldots,p-1\}$) and $p<k$, then define $\beta$ and $\gamma$, respectively as:
 \begin{eqnarray}\label{n3}
        \left(\begin{array}{ccccccc}
            A_1  & \cdots & A_i& c & A_{i+1} & \cdots & A_p \\
            1  & \cdots &  i & i+1 &  i+2  & \cdots &  p+1
          \end{array} \right)
\textrm{ and }   \nonumber \\ \left(\begin{array}{ccccccc}
            1   & \cdots &  i &  i+2  & \cdots &  p+1 & p+2  \\
            k  & \cdots & k-i+1& k-i & \cdots & k-p+1 & k-p
          \end{array} \right);\end{eqnarray}
if $\max A_i<c< \min A_{i+1}$ ($i\in\{1,2,\ldots,p-1\}$) and $p=k$, then define $\beta$ and $\gamma$, respectively as:
 \begin{eqnarray}\label{n4}
        \left(\begin{array}{ccccccc}
            A_1  & \cdots & A_i& c & A_{i+1} & \cdots & A_p \\
            2  & \cdots &  i+1 & i+2 &  i+3  & \cdots &  p+2
          \end{array} \right)
\textrm{ and }  \nonumber\\ \left(\begin{array}{ccccccc}
            1  &  2  &  \cdots &  i+1 &  i+3  & \cdots & p+2  \\
            k+1 & k & \cdots & k-i+1& k-i & \cdots & k-p+1
          \end{array} \right);\end{eqnarray}
   if $\min A_i<c< \max A_{i}$ ($i\in\{1,2,\ldots,p\}$) and $p<k$, then define $\beta$ and $\gamma$, respectively as:
 \begin{eqnarray}\label{n31}
        \left(\begin{array}{ccccccc}
            A_1  & \cdots & \{\min A_i,\ldots,c-1\}& \{c+1,\ldots,\max A_i\} & A_{i+1} & \cdots & A_p \\
            1  & \cdots &  i & i+1 & i+2  & \cdots &  p+1
          \end{array} \right)
\textrm{ and }  \nonumber \\ \left(\begin{array}{ccccccc}
            1   & \cdots &  \{i,i+1\} & i+2 & \cdots &  p+1 & p+2  \\
            k  & \cdots & k-i+1 & k-i & \cdots & k-p+1 & k-p
          \end{array} \right);\end{eqnarray}
   if $\min A_i<c< \max A_{i}$ ($i\in\{1,2,\ldots,p\}$) and $p=k$, then define $\beta$ and $\gamma$, respectively as:
 \begin{eqnarray}\label{m41}
        \left(\begin{array}{ccccccc}
            A_1  & \cdots & \{\min A_i,\ldots,c-1\}& \{c+1,\ldots,\max A_i\} & A_{i+1} & \cdots & A_p \\
            2  & \cdots &  i+1 & i+2 & i+3 & \cdots &  p+2
          \end{array} \right)
\textrm{ and }  \nonumber \\ \left(\begin{array}{cccccccc}
            1  & 2  & \cdots &  \{i+1,i+2\} & i+3 & \cdots &  p+1 &p+2 \\
            k+1 & k & \cdots & k-i+1 & k-i & \cdots & k-p+2 & k-p+1
          \end{array} \right);\end{eqnarray}

 if $\max A_p<c$ and $p<k$, then we define $\beta$ and $\gamma$, respectively as:
 \begin{eqnarray}\label{n5}
        \left(\begin{array}{ccccc}
         A_1 & A_2 & \cdots & A_p & c \\
         1  &  2  &  \cdots & p&  p+1
          \end{array} \right) \textrm{ and } \left(\begin{array}{cccccccc}
            2  &  3   & \cdots &  p & p+2  \\
            k & k-1 & \cdots & k-p+1 & k-p
          \end{array} \right)\end{eqnarray}
 and if $\max A_p <c$ and $p=k$, then define $\beta$ and $\gamma$, respectively as:
 \begin{equation}\label{n6}
        \left(\begin{array}{ccccc}
         A_1 & A_2 & \cdots & A_p & c\\
          2  &  3 &  \cdots &  p+1& p+2
          \end{array} \right)
\textrm{ and }\left(\begin{array}{cccccccc}
            1  &  2  & 3 & \cdots &  p+1   \\
            k+1 & k & k-1 & \cdots & k-p+1
          \end{array} \right).\end{equation}
Then, it is easy to see that $\beta,\gamma\in W_{p+1}$, and in all the cases one can see that $\alpha=\beta\gamma$.

{\bf Case II:} If $\textnormal{im }\alpha$ is not convex. Then $\alpha$ is of the form
 $  \left(\begin{array}{cccc}
         A_1 & A_2 & \cdots & A_p  \\
          b_p  & b_{p-1} &  \cdots &  b_1
          \end{array} \right)$ where $b_1<b_2<\ldots < b_p$ and $b_i - b_{i+1}\geq 2$ for some $i\in\{1,2,\ldots,p-1\}$. By contractive property of $\alpha$ we must have $\min A_{i+1}-\max A_i\geq 2$ for such $i$. Now, define
$$\beta=\left(\begin{array}{cccccccc}
            A_1 & A_2 & \cdots & A_i& \max A_i +1 & A_{i+1} & \cdots & A_p \\
            b_p & b_{p-1} & \cdots & b_{p-i+1}& b_{p-i+1}-1 & b_{p-i} & \cdots & b_1
 \end{array} \right).$$ Observe that since $|\textnormal{im }\alpha|\leq n-2$, then $|\textnormal{im }\beta|\leq n-1$.  Let $c\in [n]\backslash \textnormal{im }\beta$ and let $B=\textnormal{im }\beta \cup \{c\}$. Define $\gamma$ to be the partial identity on $B$. Then one can easily verify that $\beta$ and $\gamma$ are in $W_{p+1}$ and $\alpha=\beta\gamma$, as required.

\end{proof}
The following corollary is immediate:
\begin{corollary}\label{cjp} For $n\geq 3$, $\langle W_{n-1}\rangle=\mathcal{ORCP}_{n,n-1}$.
\end{corollary}
\begin{proof}
  The proof follows from Lemma~\ref{jp} and the fact that $\langle W_{n-1} \rangle \leq \mathcal{ORCP}_{n-1}$.
\end{proof}

Recall the definitions of  $\alpha_{s,t}$ and $\pi_{\{i,i+1\},k}$ (for some $i,k,s,t$ in $\{1,2,\ldots n\}$) after Corollary~\ref{cjr}, let $\alpha^*_{s,t}$ and $\pi^*_{\{i,i+1\},k}$ denote the order-reversing counterparts of $\alpha_{s,t}$ and $\pi_{\{i,i+1\},k}$, respectively (which exist if $t\in \{1,n\}$ $(1\leq s\leq n)$, or $t=s-n+1$ for some $s\in \{2,3,\ldots, n-1$). For the purpose of illustration, consider the transformation $\alpha_{2,1}=  \left(\begin{array}{ccccc}
         1 & 3 & 4 &\cdots & n \\
          2 & 3 & 4 & \cdots & n
          \end{array} \right)$. Then, $\alpha^*_{2,1}=  \left(\begin{array}{ccccc}
         1 & 3 & 4 &\cdots & n \\
          n  & n-1 & n-2 & \cdots & 2
          \end{array} \right)$.

We now have elements of $W_{n-1}$ displayed in the table below, where $1\leq i\leq n-1$ and $2\leq j\leq n-1$.
\begin{equation}
\begin{tabular}{|c|c|c|c|c|}
\hline
  {\bf Kernel Partition / Image sets} & $[n]\backslash \{1\}$  & $[n]\backslash \{j\}$ & $[n]\backslash \{n-j+1\}$ & $[n]\backslash \{n\}$ \\
  \hline
  \multirow{2}{*}{$P_{[n]\backslash \{1\}}$} & $\alpha_{1,1}$ &&& $\alpha_{1,n}$    \\
                                     & $\alpha_{1,1}^*$ &&& $\alpha_{1,n}^*$   \\  \hline
  \multirow{2}{*}{$P_{[n]\backslash \{j\}}$} & $\alpha_{j,1}$ &$\alpha_{j,j}$&& $\alpha_{j,n}$    \\
                            & $\alpha_{j,1}^*$ &&$\alpha_{j,n-j+1}^*$& $\alpha_{j,n}^*$      \\\hline
\multirow{2}{*}{$P_{[n]\backslash \{n\}}$}  & $\alpha_{n,1}$ &&& $\alpha_{n,n}$ \\
                                 & $\alpha_{n,1}^*$ &&& $\alpha_{n,n}^*$   \\ \hline
  \multirow{2}{*}{$P_{\{i,i+1\}}$} & $\pi_{(i,i+1),1}$ &&& $\pi_{(i,i+1),n}$   \\
             & $\pi_{(i,i+1),1}^*$  &&& $\pi_{(i,i+1),n}^*$   \\
  \hline
\end{tabular}\tag{2}\label{tab2}
\end{equation}

 Notice that by Theorem~\ref{zbt} each row in Table~\eqref{tab2} represents a unique $\mathcal{R}^*-$class in $W_{n-1}$, while columns represents $\mathcal{L}^*-$classes. Therefore, the number of $\mathcal{R}^*-$classes in $W_{n-1}$ is $2n-1$. The following lemmas follows immediately.

 \begin{lemma}\label{rw} For $n\geq 3$, \textnormal{rank(}$\langle W_{n-1}\rangle)\geq 2n-1$.
 \end{lemma}
\begin{proof}
 The proof is similar to the proof of Corollary~\ref{lr}. \end{proof}

 \begin{lemma}\label{gs2} For $n\geq 3$, $K_{n-1}\cup\{\alpha_{n,n}^*, \alpha_{1,1}^*, \alpha_{j,n-j+1}^*: \, 2\leq j\leq n-1\}$ generates $W_{n-1}$.
 \end{lemma}
\begin{proof}
Using routine multiplication we see that
  \begin{eqnarray*}
        \alpha_{1,n}\cdot \alpha_{n,n}^* &=& \alpha_{1,n}^*, \\
         \alpha_{j,n}\cdot \alpha_{n,n}^* &=& \alpha_{j,n}^*, \\
        \alpha_{n,1}\cdot \alpha_{1,1}^* &=& \alpha_{n,1}^*, \\
         \alpha_{j,1}\cdot \alpha_{1,1}^* &=& \alpha_{j,1}^*, \\                                                                   \pi_{\{i,i+1\},n}\cdot \alpha_{n,n}^* &=& \pi_{\{i,i+1\},n}^* \textrm{ and}\\
         \pi_{\{i,i+1\},1}\cdot \alpha_{1,1}^* &=& \pi_{\{i,i+1\},1}^*.
                                                                 \end{eqnarray*}
  Therefore, $W_{n-1}\subseteq \langle K_{n-1}\cup\{\alpha_{n,n}^*, \alpha_{1,1}^*, \alpha_{j,n-j+1}^*\} \rangle$. Hence, $K_{n-1}\cup\{\alpha_{n,n}^*, \alpha_{1,1}^*, \alpha_{j,n-j+1}^*\} $  generates $W_{n-1}$.
\end{proof}

The next proposition gives us a minimal generating set of $W_{n-1}$.

\begin{proposition}\label{mgs} Let $ M=\{\alpha_{1,n}, \pi_{\{i,i+1\},n},\alpha_{j,n-j+1}^*: 1\leq j\leq n,\, 1\leq i\leq n-1 \}$. Then $M$ is a minimal generating set of $W_{n-1}$.
\end{proposition}
\begin{proof}
  By Lemma~\ref{rw}, if $B$ is any generating set of $W_{n-1}$, then $B$ must have at least $2n-1$ elements. Furthermore, notice that $|M|=2n-1$. Thus, to show $M$ is a minimal generating set of $W_{n-1}$ we only need to show that $M$ generates $W_{n-1}$. Let $ Z=\{\alpha_{n,n}^*, \alpha_{1,1}^*, \alpha_{j,n-j+1}^*: \, 2\leq j\leq n-1\}$. Then, by Lemma~\ref{gs2}, $K_{n-1}\cup\ Z $ generates $W_{n-1}$. Now, recall that by Lemma~\ref{gs1}, the set $G=\{\alpha_{n,1},\alpha_{1,n},\alpha_{j,j},\pi_{\{i,i+1\},n}: 2\leq j\leq n-1,\, 1\leq i\leq n-1 \}$ generates $K_{n-1}$. Therefore, $G \cup Z=\{\alpha_{n,1},\alpha_{1,n},\alpha_{j,j},\pi_{\{i,i+1\},n},\alpha_{n,n}^*, \alpha_{1,1}^*, \alpha_{j,n-j+1}^*: 2\leq j\leq n-1,\, 1\leq i\leq n-1 \}$ generates $W_{n-1}$.

  However, not all elements in $G \cup Z$ are needed in order to generate $W_{n-1}$, for, if $ \alpha_{n,1}$ is replaced with $ \alpha_{n,1}^*$, then we observe that
   \begin{eqnarray*}
        \alpha_{n,1}^*\cdot \alpha_{1,n} &=& \alpha_{n,n}^*, \\
         \alpha_{1,n}\cdot \alpha_{n,1}^* &=& \alpha_{1,1}^*, \\
         \alpha_{j,n-j+1}^*\cdot\alpha_{n-j+1,j}^* & = & \alpha_{j,j} \textnormal{ and}\\
         \alpha_{n,1}^*\cdot\alpha_{1,n}\cdot\alpha_{n,1}^* &= & \alpha_{n,1}.
                                                                 \end{eqnarray*}
  Meaning that $\alpha_{n,n}^*, \alpha_{1,1}^*, \alpha_{n,1}$  and $\alpha_{j,j}\, (2\leq j\leq n-1)$ are redundant elements.

  Hence, $\{\alpha_{1,n}, \pi_{\{i,i+1\},n},\alpha_{j,n-j+1}^*: 1\leq j\leq n,\, 1\leq i\leq n-1 \}$ generates $W_{n-1}$ as required.
 \end{proof}

\begin{corollary}\label{rwn} For $n\geq 3$, \textnormal{rank}$\langle W_{n-1}\rangle= 2n-1$.
\end{corollary}
\begin{proof}
  Follows from Proposition~\ref{mgs}.
\end{proof}
We now have the rank of $\mathcal{OCP}_{n,n-1}$ in the following proposition.

\begin{proposition}\label{rw2} \textnormal{Rank(}$\mathcal{OCP}_{n,n-1}) =2n-1$ for all $n\geq 3$.
\end{proposition}

\begin{proof} Follows from Corollaries~\ref{cjp} and~\ref{rwn}.
\end{proof}

 Next, it is not difficult to see that the top $\mathcal{J}$-class of $\mathcal{ORCP}_{n}$ (i. e., the $J_{n}$ class) contain only the elements, $1_{[n]}$ and $1^*_{[n]}$, where $1_{[n]}= \left(\begin{array}{ccccc}
         1 & 2 & 3 &\cdots & n\\
          1  & 2 & 3 & \cdots & n
          \end{array} \right)$. Thus, $\mathcal{ORCP}_{n}$ is a disjoint union of $\mathcal{ORCP}_{n,n-1}$ and $\{1_{[n]},1^*_{[n]}\}$. Consequently, we have

 \begin{theorem}
  Let $\mathcal{ORCP}_{n}$ be as defined in equation~\eqref{orcp}. Then, the rank of $\mathcal{ORCP}_{n}$ is $2n$ for all $n\geq 3$.
 \end{theorem}
\begin{proof}
  First, notice that since $\mathcal{ORCP}_{n,n-1}$ is an ideal of $\mathcal{ORCP}_{n}$, then it cannot generate any element in the upper $\mathcal{J}$-class. Furthermore, by simple multiplication one can show that $(1^*_{[n]})^2= 1_{[n]}$, $(1_{[n]})^2=1_{[n]}$, $1^*_{[n]} 1_{[n]} =1^*_{[n]}$ and $1_{[n]} 1^*_{[n]} =1^*_{[n]}$. Thus, $1^*_{[n]}$ generates $\{1_{[n]},1^*_{[n]}\}$ and $\{1_{[n]},1^*_{[n]}\}$ generates no other element in the lower $\mathcal{J}$-class. Consequently, we have rank($\mathcal{ORCP}_{n})=\textnormal{rank}(\mathcal{ORCP}_{n,n-1})+ \textnormal{rank}(\{1_{[n]},1^*_{[n]}\})$, which is $2n$.
\end{proof}

\end{document}